\documentclass[12pt,a4paper]{amsart}
\usepackage{amsfonts}
\usepackage{amsthm}
\usepackage{amsmath}
\usepackage{amscd}
\usepackage[latin2]{inputenc}
\usepackage{t1enc}
\usepackage[mathscr]{eucal}
\usepackage{indentfirst}
\usepackage{graphicx}
\usepackage{graphics}
\usepackage{pict2e}
\usepackage{epic}
\numberwithin{equation}{section}
\usepackage[margin=2.9cm]{geometry}
\usepackage{epstopdf} 
\usepackage{enumerate}
\usepackage{verbatim}
\usepackage{cite}
\usepackage{mathtools}
\usepackage{mathrsfs}
\usepackage{xcolor}
\usepackage{hyperref}

\DeclarePairedDelimiter\floor{\lfloor}{\rfloor}

\theoremstyle{plain}
\newtheorem{Th}{Theorem}[section]
\newtheorem{Lemma}[Th]{Lemma}

\theoremstyle{definition}

\newtheorem{Rem}[Th]{Remark}
\newtheorem{Ex}[Th]{Example}

\definecolor{ao(english)}{rgb}{0.0, 0.5, 0.0}

\begin{document}

\title[Counting essential minimal surfaces in n-manifolds]{Counting essential minimal surfaces in closed negatively curved n-manifolds}

\author{Ruojing Jiang}

\address{University of Chicago, Department of Mathematics, Chicago, IL 60637} 
\email{ruojing@math.uchicago.edu}

 \subjclass[2019]{} %页尾
 \date{}
 %\keywords{sample paper} 

\begin{abstract} For closed odd-dimensional manifolds with sectional curvature less or equal than $-1$, we define the minimal surface entropy that counts the number of surface subgroups. 
It attains the minimum if and only if the metric is hyperbolic. Moreover, we compute the entropy associated with other locally symmetric spaces and cusped hyperbolic 3-manifolds. 
\end{abstract}

\maketitle

\section{Introduction}

Let $M$ be a closed manifold with negative sectional curvature, a classical theorem states that among all metrics on $M$ with the same volume, the entropy counting the number of closed geodesics attains the minimum if and only if the metric is locally symmetric. Recently, Calegari-Marques-Neves \cite{Calegari-Marques-Neves} focused on closed three-dimensional manifolds admitting hyperbolic metrics and studied the counting property for minimal surfaces, which shifted attention from the one-dimensional case (geodesics) to two dimensions. 
Based on Hamenst{\"a}dt's construcion of surface subgroups in \cite{Hamenstadt}, they defined an entropy $E(h)$ to measure the number of related essential minimal surfaces with respect to the metric $h$ of sectional curvature at most $-1$. The entropy achieves the minimum value $2$ if and only if $h$ is hyperbolic. In this paper, we continue to investigate the extent to which the definition of entropy and the rigidity property hold for generalized ambient manifolds.

Now suppose $M=\mathbb{H}^n/\Gamma$ is a closed orientable $n$-manifold $(n\geq 3)$ that admits a hyperbolic metric $h_{hyp}$. 
Let $S$ be any closed surface immersed in $M$ with genus at least $2$, $S$ is \emph{essential} if the immersion is $\pi_1$-injective, and the image of $\pi_1(S)$ in $\pi_1(M)$ is called a \emph{surface subgroup}. 
Let $S(M,g)$ denote the set of surface subgroups of genus at most $g$ up to conjugacy, and let the subset $S(M,g,\epsilon)\subset S(M,g)$ consist of the conjugacy classes whose limit sets are $(1+\epsilon)$-quasicircles (see the definition below). Moreover, \begin{equation*}
    S_\epsilon(M)=\underset{g\geq 2}{\cup}S(M,g,\epsilon).
\end{equation*}

\subsection{Definitions}\label{1.1}
A homeomorphism $f:S^{n-1}_\infty\rightarrow S^{n-1}_\infty$ is called \emph{$K$-quasiconformal} if $K(f)=\underset{x\in S^{n-1}_\infty}{\text{ess sup}} K_f(x)\leq K$, where the \emph{dilatation} of $f$ is \begin{equation*}
        K_f(x)=\underset{r\rightarrow 0}{\lim\sup}\,\underset{y,z}{\sup}\frac{|f(x)-f(y)|}{|f(x)-f(z)|},
    \end{equation*}
    the supremum is taken over $y,z$ so that $r=|x-y|=|x-z|$.
    
   % Equivalently, let $\lambda_1\leq\cdots\leq\lambda_n$ be the eigenvalues of $D_xf$, then $K_f(x)=\frac{\lambda_n}{\lambda_1}$. Then $f$ is $K$-quasiconformal if for every $x$, $D_xf$ sends a sphere to an ellipsoid of eccentricity bounded by $K$.

A \emph{$K$-quasicircle} in $S^{n-1}_\infty$ is the image of a round circle under a $K$-quasiconformal map.

Now suppose $h$ is an arbitrary Riemannian metric on $M$. For any $\Pi\in S(M,g)$, we set \begin{equation*}
    \text{area}_h(\Pi)=\inf \{\text{area}_h(\Sigma):\Sigma\in\Pi\}.
\end{equation*}
Then the minimal surface entropy with respect to $h$ is defined as follows. 
\begin{equation}\label{eh}
    E(h)=\underset{\epsilon\rightarrow 0}{\lim}\,\underset{L\rightarrow\infty}{\lim\inf}\,\dfrac{\ln\#\{\text{area}_h(\Pi)\leq 4\pi(L-1):\Pi\in S_\epsilon(M)\}}{L\ln L}.
\end{equation}
This quantity measures the exponential growth rate of the number of surface subgroups, by Schoen-Yau \cite{Schoen-Yau} and Sacks-Uhlenbeck \cite{Sacks-Uhlenbeck}, each surface subgroup corresponds to an immersed essential surface which minimizes area in homotopy class, and we'll see that it is unique provided that $|A|_\infty^2<2$.

\subsection{Main theorem}
The main theorem related to $E(h)$ is the following.
\begin{Th}\label{main}
Let $M$ be a closed orientable manifold whose dimension $n\geq 3$ is odd, and $M$ admits a hyperbolic metric $h_{hyp}$.
\begin{enumerate}[(a)]
    \item Given an arbitrary metric $h$ on $M$, we have $E(h)\leq 2E_{\text{vol}}(h)^2$.
    
    \item If the sectional curvature of $h$ is less than or equal to $-1$, then \begin{equation*}
        E(h)\geq E(h_{hyp})=2.
    \end{equation*}
    Furthermore, if the equality holds, then $h$ is isometric to $h_{hyp}$.
\end{enumerate}
\end{Th}

\begin{comment}
\begin{Th}\label{theorem2}
Let $M$ be a closed orientable manifold of dimension $n\geq 3$ which admits a hyperbolic metric $h_{hyp}$. Then there is a neighborhood $\mathcal{V}$ of $h_{hyp}$, such that if $h$ is a metric on $M$ satisfying that $R(h)\geq -n(n-1)$, then \begin{equation*}
    E(h)\leq E(h_{hyp})=2,
\end{equation*}
the equality holds if and only if $h$ is isometric to $h_{hyp}$.
\end{Th}
\end{comment}

\subsection{Sketch of proof}
(a) follows directly from the proof in \cite{Calegari-Marques-Neves}. In this paper, we prove (b) in several steps. Suppose that $\Pi\in S_\epsilon(M)$, and let $S$ be the essential minimal surface of $M$ in the homotopy class $\Pi$ so that $\text{area}(S)=\text{area}(\Pi)$. In Section \ref{section2}, we show that when $\epsilon$ is sufficiently small, $S$ is free of branch points, and its second fundamental form satisfies that \begin{equation}\label{A}
        |A|^2_{L^\infty(S)}=o_\epsilon(1).
    \end{equation}
Then we compute $E(h_{hyp})$ in Section \ref{section3}. And the inequality and the rigidity in (b) is shown later in Section \ref{section_rigidity}.

Additionally, in Section \ref{section_symmetricspace} and Section \ref{section_cusp}, we briefly investigate the entropy for other locally symmetric spaces and cusped hyperbolic 3-manifolds.

\subsection*{Acknowledgements}
I would like to thank my advisor Andr\'{e} Neves for his continued support and many useful suggestions.

\section{Proof of (\ref{A})}\label{section2}
In this section, we prove (\ref{A}). More specifically, Let $M=\mathbb{H}^n/\Gamma$ be a closed hyperbolic manifold with dimension $n\geq 4$. 
From the argument by Schoen-Yau \cite{Schoen-Yau} and Sacks-Uhlenbeck \cite{Sacks-Uhlenbeck}, for any surface subgroup $\Pi<\Gamma$, there exists an immersed minimal surface $S$ in $M$ with finitely many branch points, such that it minimizes the area in the corresponding homotopy class up to conjugacy. 
The appearance of branch points is the primary distinction from the case of dimension three. So the key point of (\ref{A}) is to rule out branch points.

\subsection{Trap of convex hull}\label{subsection_barrier}

First of all, we need to argue that the convex hull of $\Pi$ lies in a bounded region. In \cite{Seppi}, Seppi proved that for any minimal disc $D\subset \mathbb{H}^3$ asymptotic to a $(1+\epsilon)$-quasicircle $\gamma$, every point $x$ in $D$ lies on a geodesic segment $\alpha$ that is orthogonal at the endpoints to planes $P_+$ and $P_-$, such that the convex hull of $\gamma$ is bounded between $P_+$ and $P_-$. Additionally, the length of $\alpha$ goes to zero uniformly in $x$ as $\epsilon$ approaches zero. In the following lemma, we construct the bound of the convex hull of $\Pi$ while the ambient manifold has dimension at least $4$.

\begin{Lemma}\label{lem_barrier}
Given $\Pi\in S_\epsilon(M)$ and let $\gamma$ be the limit set of $\Pi$.
Then the convex hull of $\gamma$ is contained in a tube, %whose diameter converges to zero on compact sets as $\epsilon$ goes to zero.
which converges in Hausdorff distance as $\epsilon$ goes to zero, the limit is either empty, or it is contained in a totally geodesic disc.
\end{Lemma}

\begin{proof}
Since $\gamma$ is an $(1+\epsilon)$-quasicircle, we can find $\epsilon'=o_\epsilon(1)$ and a round circle $c\subset S^{n-1}_\infty$, such that the $\epsilon'$-neighborhood of $c$ in $S^{n-1}_\infty$, denoted by $N$, contains $\gamma$. Any round circle in $\partial N$ bounds a unique totally geodesic disc in $\mathbb{H}^n$, and thus there is a tube $T\subset \mathbb{H}^n$ homeomorphic to $D^2\times S^{n-3}$ asymptotic to $\partial N$. Notice that as $\epsilon$ goes to zero, the limit of $T$ is either empty, or it is contained in a totally geodesic disc. 
In addition, $T$ is convex, so the convex hull of $\gamma$ is bounded in $T$.
\end{proof}

\subsection{Uniqueness of minimal surfaces}

\begin{Lemma}\label{uniqueness}
Let $D$ be a minimal surface in $\mathbb{H}^n$ asymptotic to $\gamma\subset S_\infty^{n-1}$, where $\gamma$ is the limit set of a surface subgroup of $M$. If the norm squared of the second fundamental form of $D$ satisfies that $|A|_{L^\infty(D)} ^2<2$, then \begin{enumerate}[(a)]
\item $D$ is an embedded disc, \label{item_a}
\item If $D$ is the lift of a closed surface in $M$, then $\gamma=\partial_\infty D$ is a quasicircle, \label{item_b}
\item Assume that $\gamma$ is a quasicircle, then $D$ is the unique minimal surface with $\partial_\infty D=\gamma$ of all types. \label{item_c}
\end{enumerate}
\end{Lemma}

The proof is similar to that of Uhlenbeck's argument \cite{Uhlenbeck}, but we need to be careful about the noncompactness of $D$, so it's worth providing the proof here.
\begin{proof}
First, (\ref{item_a}) can be shown by finding a horosphere $S$ that touches $D$ at a single point, and suppose $\partial_\infty S=y$, namely $S$ is centered at $y$. Notice that $\mathbb{H}^n$ is foliated by horospheres centered at $y$, and since the principal curvatures of $D$ restricted to any normal direction are less than that of the horospheres, the distance from a fixed point in $\mathbb{H}^n\setminus D$ to the points on $D$ is a convex function, so there's only one critical point which attains the minimum. This is impossible when the intersection of $D$ with an extrinsic ball is an annulus or it has self-intersections, so it forces $D$ to be an embedded disc.

Furthermore, we can also prove that $\exp|_{ND}$ is a diffeomorphism from $ND$ to $\mathbb{H}^n$, and therefore $\mathbb{H}^n\setminus D$ is foliated by a family of hypersurfaces $\{H^r\}_{r>0}$, where $H^r$ is the hypersurface at the fixed distance $r$ to $D$. 
To see this, we need some notations beforehand.
For $x\in D$, choose an oriented, orthonormal basis $\{e_1,e_2\}$ for $T_{x}D$, and an orthonormal basis $\{e_3,\cdots,e_n\}$ for $N_{x}D$. Then we obtain an orthonormal frame by applying the parallel transport along $\exp|_{ND}$. Since $D$ is a minimal surface, for any $3\leq j\leq n$, let $\lambda_{ijk}:=\langle A(e_j,e_k),e_j\rangle$, then we have 
$\lambda_{j11}=-\lambda_{j22}:=\lambda_j$. 
We assume $\lambda_j\geq 0$ in the following computation. 
Moreover, for any $3\leq j\leq n$ and $i\neq j $, let $V_{ji}(r)=v^i_j(r)e_i(r)$ be the Jacobi field along $\exp(re_j)$. When $1\leq i\leq 2$, it satisfies that $v^i_j(0)=1$, $(v^i_j)'(0)=\lambda_{jii}(x)=\pm \lambda_j(x)$. And when $3\leq i\leq n$, $i\neq j$, we take $v^i_j(0)=0,$ $(v^i_j)'(0)=1$. Then \begin{equation*}
    V_{ji}(r)=
    \begin{cases}
    (\cosh{r}\pm \lambda_j\sinh{r})e_i(r) & \text{if $1\leq i\leq 2$}\\
    \sinh{r}e_i(r) & \text{if $3\leq i\leq n$ and $i\neq j$}
    \end{cases}
\end{equation*}
So the induced metric on $H^r$ is 
\[
g_j(x,r)=
\begin{bmatrix}
(\cosh{r}+\lambda_j(x)\sinh{r})^2 & \\
& (\cosh{r}-\lambda_j(x)\sinh{r})^2 &\\
& & \sinh^2{r}& \\
& & & \ddots &\\
& & & & \sinh^2{r}

\end{bmatrix}
\]
It is nonsingular for any $r>0$ since $\lambda_j(x)< 1$. In addition, we've seen that $\exp|_{ND}: ND\rightarrow \mathbb{H}^n$ is a bijection because of the convexity of the distance function, thus it is a diffeomorphism and $\mathbb{H}^n$ admits a foliation structure.

Additionally, assume that $D$ is the lift of a closed surface $S\subset M$. Since the induced metric on $S$ is conformally equivalent to a hyperbolic metric $g_{hyp}$, the induced metric on $D\subset\mathbb{H}^n$ is also conformally equivalent to a hyperbolic metric, we still denote it by $g_{hyp}$. Because $\lambda_j<1$, by \cite{Uhlenbeck}, the induced metric on $H^r$ is quasi-isometrically equivalent to the metric 
\[
\begin{bmatrix}
\cosh^2{r}\,g_{hyp} & \\
& \sinh^2{r}\,I_{(n-3)\times(n-3)}

\end{bmatrix}
\]
The quasi-isometry  $F: \mathbb{H}^n\rightarrow\mathbb{H}^n$ extends to a quasiconformal map $f: S^{n-1}_\infty\rightarrow S^{n-1}_\infty$.
It means that $\gamma$ is the image of a round circle mapped by $f$, so equivalently, it is a quasicircle. This proves (\ref{item_b}).

Next, to prove (\ref{item_c}), we denote by $\lambda_{ji}(x,r)$ the principal curvature of $H^r$. Then
\begin{align*}
    \lambda_{j1}(x,r) &= \frac{\tanh r+\lambda_j(x)}{1+\lambda_j(x)\tanh r}>0,\\
    \lambda_{j2}(x,r) &= \frac{\tanh r-\lambda_j(x)}{1-\lambda_j(x)\tanh r},\\
    \lambda_{ji}(x,r) &= \frac{1}{\tanh r}>0,\quad\quad \forall\, 3\leq i\leq n \,\text{ and }\, i\neq j.%\nonumber
\end{align*}
$\lambda_{j2}(x,r)$ is the only one that is possibly non-positive. 
The assumption $\lambda_j(x)< 1$ yields that \begin{align*}
    & \lambda_{j1}(x,r)+\lambda_{j2}(x,r)=\frac{2(1-\lambda_j^2(x))\tanh{r}}{1-\lambda_j^2(x)\tanh^2{r}}>0,\\
    & \lambda_{j2}(x,r)+\lambda_{ji}(x,r)=\frac{\tanh^2{r}-2\lambda_j(x)\tanh{r}+1}{(1-\lambda_j(x)\tanh{r})\tanh{r}}>\frac{1-\lambda_j(x)\tanh{r}}{\tanh{r}}>0, 
\end{align*}
where $3\leq i\leq n$ and $i\neq j$.
Therefore, for any $r>0$, $H^r$ is strictly two-convex. Furthermore, when $r>\tanh^{-1}\frac{|A|^2_{L^\infty(D)}}{2}$, we have $r>\tanh^{-1} |\lambda_j|_{L^\infty(D)}$ for all $3\leq j\leq n$, then all the principal curvatures of $H^r$ are positive, thus $H^r$ is strictly convex and it bounds inside the convex hull of $\gamma$.

Now let $D'$ be any other minimal surface with $\partial_\infty D'=\gamma$, and let $R>0$ be the supremum of $r$ such that $D'$ intersects $H^r$. From \cite{Anderson}, $D'$ lies in the convex hull of $\gamma$, then $R$ cannot exceed the finite number $\tanh^{-1}\frac{|A|^2_{L^\infty(D)}}{2}$. If the supremum is attained on $D'$, then $D'$ is tangent to the two-convex hypersurface $H^R$, which contradicts the maximum principle. 

Otherwise, if the supremum $R$ is not attained. We can take a sequence $x_i'\in D'\cap H^{r(x_i')}$, such that $r(x_i')\rightarrow R$ as $i\rightarrow \infty$. For each $x_i'$, there exists an isometry $T_i$ of $\mathbb{H}^n$ sending $x_i'$ to the origin. In the meantime, since $\gamma$ is a quasicircle, there is a quasiconformal map $\phi_i$ on $S_\infty^{n-1}$ that maps a round circle to $T_i(\gamma)$.
%, such that $\phi_i(S^1)=T_i(\gamma)$. 
And since the convex hull of $T_i(\gamma)$ contains the origin, there exist three distinct points $x,y,z$ on the round circle, so that
$\phi_i(x), \phi_i(y), \phi_i(z)$ are at a uniformly positive distance from one another, it then follows from the compactness theorem of quasiconformal maps that after passing to a subsequence, $\phi_i$ converges uniformly to a quasiconformal map, and $T_i(\gamma)$ converges to a quasicircle $\gamma_\infty$ in Hausdorff topology (see, for instance, page 25 in \cite{Kapovich}).
Moreover, 
after passing to a subsequence, $T_i(D')$ converges as varifolds to $D'_\infty$ with $\partial_\infty D'_\infty=\gamma_\infty$. We can further assume that $T_i(D)$ converges to $D_\infty$ with $\partial_\infty D_\infty=\gamma_\infty$ (since the points $x_i$, which are the normal projections of $x_i'$ to $D$, are mapped into a compact region by $T_i$, we can take further isometries sending $T_i(x_i)$ to the origin and repeat the same procedure). Standard compactness theorem implies that $T_i(D)$ converges graphically, thus the limit $D_\infty$ is still an embedded minimal surface with $|A|^2_\infty<2$. Denote by $H_\infty^r$ the hypersurface at the fixed distance $r$ to $D_\infty$, then $\{H_\infty^r\}_{r>0}$ foliates $\mathbb{H}^n\setminus D_\infty$. Moreover, $D'_\infty$ is tangent to the two-convex hypersurface $H^R_\infty$, but this violates the maximum principle.

\end{proof}

\subsection{Absence of branch points}\label{absence}

From now on, we consider a sequence $\Pi_i\in S(M,g_i,\frac{1}{i})$, and let $S_i$ be a minimal surface in the homotopy class $\Pi_i$ such that $\text{area}(S_i)=\text{area}(\Pi_i)$. We denote by $\gamma_i$ the limit set of $\Pi_i$.

Let $D_i\subset\mathbb{H}^n$ be the area-minimizing surface mod 2 with $\partial_\infty D_i=\gamma_i$. Then due to Theorem 3 of \cite{Almgren66}, $D_i$ is free of branch points. Moreover, after applying an isometry in $Isom(\mathbb{H}^n)$ and passing to a subsequence, $\gamma_i$ converges to a round circle $\gamma$ in Hausdorff topology. From Corollary 2.5 in \cite{Bangert-Lang}, for any $R>0$, there exists a constant $C_R$ depending only on $n$ and $R$, such that \begin{equation}\label{area_bound}
    \text{area}(D_i\cap B_R(0))\leq C_R\quad\forall i\geq 1.
\end{equation}
It follows that $D_i$ converges as varifolds to $V$, which is also area-minimizing mod 2 (see 34.5 and 42.7 in \cite{GMT_Simon}). 
By Lemma \ref{lem_barrier}, $V$ is either empty or contained in the totally geodesic disc $D$ with $\partial_\infty D=\gamma$.

Furthermore, Alexander duality indicates that there exists a $(n-2)$-dimensional submanifold $K\subset\mathbb{H}^n$, such that the boundary $\partial K$ lies in the complement of a tubular neighborhood of $\gamma$ in $S_\infty^{n-1}$, and it is linked with every $\gamma_i$ for large enough $i$. So every $D_i$ intersects $K$. Additionally, according to Lemma \ref{lem_barrier}, there exists a radius $R_0>0$, such that \begin{equation}\label{h2}
    C(\gamma_i)\cap K\subset B_{R_0}(0),\quad\forall i>>1,
\end{equation}
where $C(\gamma_i)$ represents the convex hull of $\gamma_i$.
For any $R>R_0$ and any point $x_i\in D_i\cap B_{R_0}(0)$, we have $B_{R-R_0}(x_i)\subset B_R(0)$.  
Then (\ref{h2}), together with the monotonicity formula in \cite{Bangert-Lang}, produces a uniform constant $c_{R-R_0}>0$, such that \begin{equation}\label{lower_bound}
    \text{area}(D_i\cap B_{R}(0))\geq c_{R-R_0}.
\end{equation}
It implies that $V$ is non-empty. Thus by constancy theorem (41.1 in \cite{GMT_Simon}), $V$ is a positive multiple of $D$. And since $V$ is area-minimizing mod 2, the multiplicity has to be one.

Moreover, Allard regularity theorem (see \cite{Allard}, or Theorem 1.1 in \cite{White} for an easy version) indicates that the convergence is smooth on compact sets, and we obtain that $|A|^2_{L^\infty_{loc}(D_i)}\rightarrow 0$. 
Finally, from Lemma \ref{uniqueness}, whenever $i$ is sufficiently large, the lift of $S_i$ must coincide with $D_i$, so we finish the proof of (\ref{A}).

\section{Entropy of hyperbolic metric}\label{section3}
Let $s(M,g)$ denote the cardinality of $S(M,g)$, and $s(M,g,\epsilon)$ denote the cardinality of the subset $S(M,g,\epsilon)\subset S(M,g)$. We prove the following inequality in this section. 
\begin{equation}\label{s}
   (c_1g)^{2g}\leq s(M,g,\epsilon)\leq s(M,g)\leq (c_2g)^{2g}.
\end{equation}

\subsection{The upper bound of (\ref{s})}\label{section_upper}
When $n=3$, Kahn-Markovic \cite{Kahn-Markovic} found an upper bound of $s(M,g)$. Their method also applies to the case of $n>3$. We only need the following fact. 

Let $M$ be a closed hyperbolic manifold of dimension $n\geq 3$, and let $S_g$ denote a closed surface of genus $g$. 
For any $\pi_1$-injective immersion $f: S_g\rightarrow M$, there exist a hyperbolic structure on $S_g$ and a homotopy of $f$ that is pleated with respect to this structure, we still denote it by $f$ (see 8.10 in \cite{thurston} for Thurston's original proof for $n=3$, and Lemma 3.6 in \cite{Calegari_scl} for the generalization of all dimensions). 
Furthermore, $s(M,g)$ can be estimated by counting the number of homotopy classes of the pleated immersions. As shown in \cite{Kahn-Markovic}, there exists a constant $c_2$ depending only on the injectivity radius of $M$, so that 
\begin{equation*}
    s(M,g)\leq (c_2g)^{2g}.
\end{equation*}

\subsection{The lower bound of (\ref{s})}\label{section_lower}
Suppose that $M$ has odd dimension $n\geq 3$.
According to \cite{Hamenstadt}, for any small number $\epsilon'>0$, there is an essential surface in $M$ which is sufficiently well-distributed and $(1+\epsilon')$-quasigeodesic, namely, the geodesics on the surface with respect to intrinsic distance are $(1+\epsilon',\epsilon')$-quasigeodesics in $M$. This determines a quasi-isometry that embeds $\mathbb{H}^2$ into $\mathbb{H}^n$, whose boundary extends to a quasisymmetry $f_1: S^1_\infty\rightarrow S^{n-1}_\infty$. 
Pick two discs $D_1, D_2\subset S^{n-1}_\infty$ with $\partial D_1=\partial D_2= S^1_\infty$ and $D_1\cup D_2= S^2_\infty$.
By \cite{Tukia-Vaisala}, $f_1$ extends to quasiconformal maps from $D_i$ into $S^{n-1}_\infty$, $i=1,2$, so there exists a quasiconformal extension $f_2: S^2_\infty\rightarrow S^{n-1}_\infty$. Repeating this process,
we can find a quasiconformal extension $f_{n-1}: S^{n-1}_\infty\rightarrow S^{n-1}_\infty$. Moreover, it has dilatation $1+\epsilon$, where $\epsilon$ depends on $n$ and $\epsilon'$, and $\epsilon\rightarrow 0$ as $\epsilon'\rightarrow 0$. For this reason, we denote this essential surface by $\Sigma_\epsilon$. So for any $\epsilon>0$, we can choose a sufficiently small $\epsilon'$ to build an essential surface $\Sigma_\epsilon\subset M$ associated with an element in $S_\epsilon(M)$.
Let $G(M,g,\epsilon)$ denote the subset of $S(M,g,\epsilon)$ consisting of homotopy classes of finite covers of $\Sigma_\epsilon$ that have genus at most $g$. Counting the commensurability classes in $G(M,g,\epsilon)$ and using M{\"u}ller-Puchta's formula (see \cite{Kahn-Markovic}), we obtain the following lower bound when $g$ is large. \begin{equation*}
    s(M,g,\epsilon)\geq \#G(M,g,\epsilon)\geq (c_1g)^{2g},
\end{equation*} 
where $c_1$ is a constant that depends only on $M$ and $\epsilon$.

Moreover, let $S_i$ denote the minimal representative in the homotopy class $\Pi_i\in G(M,g_i,\frac{1}{i})$, then it is homotopic to a $(1+\frac{1}{i})$-quasigeodesic surface $\Sigma_i$. 
Assume that $\mu_i,\nu_i$ represent the Radon measures induced by integration over $S_i$ and $\Sigma_i$, respectively.
From Lemma 4.3 of \cite{Calegari-Marques-Neves}, we have \begin{equation}\label{measure}
    \underset{i\rightarrow\infty}{\lim}\mu_i=\underset{i\rightarrow\infty}{\lim}\nu_i=:\nu,
\end{equation}
Furthermore, following the argument on page 16 in \cite{Calegari-Marques-Neves}, we conclude that the measure $\nu$ is positive on any non-empty open set of $M$. The proof makes use of the property that $\Sigma_i$ is nearly equidistribued in $M$ (see $\mathsection 7$ of \cite{Hamenstadt}), and the estimates hold for all odd ambient dimensions. This measure $\nu$ plays an important role in the proof of the rigidity in Section \ref{section_rigidity}.

\subsection{Computation of the entropy}\label{subsection_b}
We prove $E(h_{hyp})=2$ first. 
Given $\eta>0$, for all sufficiently small $\epsilon$, and sufficiently large $L$ which only depend on $\eta$, we conclude from (\ref{A}) that for $\Pi\in S_\epsilon(M)$, if it has $\text{area}(\Pi)\leq 4\pi(L-1)$, then $\Pi\in S(M,\floor*{(1+\eta)L},\epsilon)$. 
On the other hand, $\Pi\in S(M,\floor*{(1-\eta)L},\epsilon)$ imples that $\text{area}(\Pi)\leq 4\pi(L-1)$.
Then consequently, \begin{align*}
    2(1-\eta)&\leq\underset{L\rightarrow\infty}{\lim\inf}\frac{\ln{s(M,\floor*{(1-\eta)L},\epsilon)}}{L\ln{L}}\\
    &\leq\underset{L\rightarrow\infty}{\lim\inf}\dfrac{\ln\#\{\text{area}_h(\Pi)\leq 4\pi(L-1):\Pi\in S_\epsilon(M)\}}{L\ln L}\\
    &\leq\underset{L\rightarrow\infty}{\lim\inf}\frac{\ln{s(M,\floor*{(1+\eta)L},\epsilon)}}{L\ln{L}}\leq 2(1+\eta).
\end{align*}
It follows directly that $E(h_{hyp})=2$.

\section{Proof of rigidity in Theorem \ref{main}}\label{section_rigidity}
The key fact to show the rigidity is the following, which is an extension of Theorem 5.1 in \cite{Calegari-Marques-Neves} to higher dimensional ambient manifolds. 

\begin{Th}\label{5.1}
Assume $M$ is defined as above. Let $S_i$ be the essential surface immersed in $M$ that minimizes area of a surface subgroup $\Pi_i <\Gamma$ in $G(M,g_i,\frac{1}{i})$ with respect to the hyperbolic metric $h_{hyp}$. And let $\Sigma_i$ be the essential surface (possibly with branch points) homotopic to $S_i$ that minimizes area with respect to the metric $h$. Then 
\begin{equation}\label{area}
    \underset{i\rightarrow\infty}{\lim\sup}\dfrac{\text{area}_h(\Sigma_i)}{\text{area}(S_i)}\leq 1.
\end{equation}
If the equality holds, then $h$ is hyperbolic and isometric to $h_{hyp}$.
\end{Th}

\subsection{Proof of Theorem \ref{5.1}}
Let's prove the inequality first. $\Sigma_i$ may have isolated branch points when $n\geq 4$, we denote by $P_i$ the locus of branch points, and by $d_{i_j}$ the order of each branch point $p_{i_j}\in P_i$. Then a generalized Gauss-Bonnet theorem in \cite{Fang96} states that \begin{equation}\label{generalizedgauss}
    \text{area}_h(\Sigma_i)=4\pi(g_i-1)+\int_{\Sigma_i\setminus P_i}(K_{12}+1)dA_h-\frac{1}{2}\int_{\Sigma_i\setminus P_i}|A|^2dA_h-2\pi\sum_jd_{i_j}.
\end{equation}
On the other hand, we have shown in Section \ref{section2} that $S_i$ has no branch points and satisfies that $|A|^2_{L^\infty(S_i)}\rightarrow 0$ as $i\rightarrow \infty$, so for sufficiently large $i$, $\text{area}(S_i)\simeq 4\pi(g_i-1)$, and thus inequality (\ref{area}) follows.

Now suppose the equality (\ref{area}) holds, it yields that \begin{equation*}
    \lim_{i\rightarrow\infty}\frac{1}{\text{area}_h(\Sigma_i)}\int_{\Sigma_i}\big(-(K_{12}+1)+\frac{1}{2}|A|^2+2\pi\sum_jd_{i_j}\delta_{p_{i_j}}\big)dA_h=0,
\end{equation*}
where $\delta_{p_{i_j}}(x)=\delta(x-p_{i_j})$ and $\delta(x)$ is the Dirac delta function.

Let $\mathscr{C}$ be the set of all round circles in $S^{n-1}_\infty$, and define \begin{align*}
\mathscr{L}=&\{\gamma\in\mathscr{C}:\exists\phi_i\in F_i(\epsilon_i,R_i), \epsilon_i\rightarrow 0, R_i\rightarrow\infty,\text{ such that}\\
& \text{after passing to subsequence, }\Lambda(\phi_i\Pi_i\phi_i^{-1})\text{ converges to }\gamma\},
\end{align*} in which \begin{equation}\label{f}
    F_i(\epsilon,R)=\{\phi\in\Gamma:\int_{\phi(\Tilde{\Sigma}_i)\cap B_R(0)}\big(-(K_{12}+1)+\frac{1}{2}|A|^2+2\pi\sum_jd_{i_j}\delta_{q_{i_j}}\big)dA_h\leq\epsilon\},
\end{equation}
where $q_{i_j}$ is a branch point with order $d_{i_j}$, the locus of branch points in $\Tilde{\Sigma}_i$ is denoted by $Q_i$, and $\delta$ represents the Dirac delta function. It's not hard to see that $\mathscr{L}$ is closed and $\Gamma$-invariant. Due to Lemma 5.2 in \cite{Shah2}, almost every element in $\mathscr{C}$ has a dense $\Gamma$-orbit. And for the elements in the subset $\mathscr{L}\subset\mathscr{C}$, we prove the following lemma.

\begin{Lemma}\label{4.2}
There is a round circle $\gamma\in\mathscr{L}$, such that $\Gamma\gamma$ is dense in $\mathscr{C}$. Additionally, it deduces a stronger result that $\mathscr{L}=\mathscr{C}$. Therefore, by \cite{Shah2}, almost every round circle in $\mathscr{L}$ has a dense $\Gamma$-orbit. 
\end{Lemma}

\begin{proof}
Theorem 6.1 of \cite{Calegari-Marques-Neves} proved the following fact using the measure $\nu$ defined in (\ref{measure}).
 \begin{align}\label{star}\tag{$\star$}
     &\text{For any $n$-dimensional compact subset $K$ of $\mathbb{H}^n$, there exists $\gamma\in\mathscr{L}$, such that}\\
     &\text{the unique totally geodesic disc $D(\gamma)$ in $\mathbb{H}^n$ bounded by $\gamma$ intersects $K$. }\nonumber
 \end{align}

Now suppose by contradiction that $\mathscr{L}$ has no element with a dense $\Gamma$-orbit in $\mathscr{C}$. According to Shah's result \cite{Shah}, for each $\gamma\in\mathscr{C}$, the projection of the $\Gamma$-orbit of $D(\gamma)$ is either dense in $M$, or it is dense in a finite union of closed totally geodesic submanifolds in $M$ of codimensions $1\leq k\leq n-2$. So for $\gamma\in\mathscr{L}$, such submanifolds are proper.
If the number of elements $\gamma\in\mathscr{L}$ so that the corresponding $D(\gamma)$ intersects $\Delta$ were finite, where $\Delta$ denote the fundamental domain of $M$, then the union of $D(\gamma)$ for all $\gamma\in\mathscr{L}$ would meet $\Delta$ in a finite subset.
Therefore, there should have been a compact subset $K\subset\Delta$ never intersecting any $D(\gamma)$, but this case can be excluded by (\ref{star}).

It turns out that $\Delta$ meets infinitely many $D(\gamma)$ for $\gamma\in\mathscr{L}$. By assumption, none of such elements $\gamma$ have dense $\Gamma$-orbits in $\mathscr{C}$, and thus the closures of the projections of these $D(\gamma)$'s in $M$ are infinitely many proper totally geodesic submanifolds, the set of such manifolds is denoted by $\mathscr{P}$.
For any $1\leq k\leq n-2$, let the subset $\mathscr{P}_k\subset\mathscr{P}$ consist of all totally geodesic submanifolds of codimensions $k$, and let $\mathscr{L}_k$ represent the collection of all $\gamma\in\mathscr{L}$ whose corresponding projection of $D(\gamma)$ in $M$ is dense in at least one of the submanifolds in $\mathscr{P}_k$. 

\begin{Lemma}\label{4.3}
All the elements in $\mathscr{P}$ are contained in a finite union of proper submanifolds of $M$.
\end{Lemma}
Let's consider an example first.
\begin{Ex}
When $n=5$, since $\mathscr{P}=\mathscr{P}_1\cup\mathscr{P}_2\cup\mathscr{P}_3$ has infinitely many candidates, we have the following cases. \begin{enumerate}[(1)]
    \item If $\mathscr{P}_1$ were infinite, then these totally geodesic submanifolds of codimension $1$ would be obviously maximal. By Corollary 1.5 of \cite{mozes_shah_1995} (or Theorem 1.7(1) of \cite{Lee-Oh}), any infinite sequence of maximal properly immersed totally geodesic submanifolds becomes dense in $M$. Thus, we could pick an infinite sequence $\{\gamma_i\}$ in $\mathscr{L}_1$,
    and then the limit of $\overline{\Gamma\gamma_i}$ should have been dense in $\mathscr{C}$. Since $\mathscr{L}$ is closed and $\Gamma$-invariant, we conclude that $\mathscr{L}=\mathscr{C}$, violating our assumption. It yields that $\mathscr{P}_1$ must be finite.
    
    \item Suppose further that $\mathscr{P}_2$ is infinite. If $\mathscr{P}_2$ contained infinitely many maximal submanifolds of $M$, then the argument in (1) could apply. So we only need to consider the situation where all but finitely many elements in $\mathscr{P}_2$ are non-maximal, denoted by $P_1, P_2,\cdots$.
    By Corollary 1.5 of \cite{mozes_shah_1995} (or Theorem 14.1(3) of \cite{Lee-Oh}), since the limit of $P_i$ doesn't become dense in $M$, any infinite subsequence of $\{P_i\}$ must have a further infinite subsequence $\{P_{j,i}\}$ contained in a proper totally geodesic submanifold $\overline{P}_j\subset M$ of higher dimension, so $\overline{P}_j$ must have codimension $1$. 
    All elements of $\{P_{j,i}\}$ are maximal submanifolds of $\overline{P}_j$, thus the limit of the sequence is dense in $\overline{P}_j$. Accordingly, there is a sequence $\{\gamma_{j,i}\}$ in $\mathscr{L}_2$, such that the closure of $\underset{i\rightarrow\infty}{\lim}\,\overline{\Gamma\gamma_{j,i}}$ contains all circles in $\mathscr{C}$ that lie in $\partial_\infty \widetilde{\overline{P}_j}\approx S^{3}$, where $\widetilde{\overline{P}_j}$ is a lift of $\overline{P}_j$ in $\mathbb{H}^5$.
    It's worth noting that
    almost every element among these circles has a dense orbit in $\partial_\infty \widetilde{\overline{P}_j}$ (Lemma 5.2, \cite{Shah2}).
    Then because of the closedness and $\Gamma$-invariance of $\mathscr{L}$,
    there exists $\overline{\gamma}_j\in\mathscr{L}$, so that the projection of $D(\overline{\gamma}_j)$ in $M$ is dense in $\overline{P}_j$. 
    In other words, we have $\overline{\gamma}_j\in\mathscr{L}_1$ and $\overline{P}_j\in \mathscr{P}_1$.

    Notice that $\mathscr{P}_1$ is a finite set, 
    we can only extract finitely many subsequences $\{P_{1,i}\},\cdots,\{P_{l,i}\}\subset\{P_i\}$ to build $\overline{P}_1,\cdots,\overline{P}_l\in\mathscr{P}_1$, the number of the remaining elements in $\mathscr{P}_2$ is finite.
    We see that $\mathscr{P}_1\cup\mathscr{P}_2$ is contained in a finite union of proper submanifolds.
    
    \item If $\mathscr{P}_1,\mathscr{P}_2$ are both finite, then $\mathscr{P}_3$ must be infinite. It suffices to assume that all but finitely many candidates in $\mathscr{P}_3$ are non-maximal, and they are denoted by $\{P_i\}$. According to \cite{mozes_shah_1995} or \cite{Lee-Oh}, any infinite subsequence has a further subsequence which is contained in a totally geodesic submanifold of codimension $1$ or $2$. 
    
    Pick a subsequence $\{P_{1,i}\}$ of $\{P_i\}$, so that the elements are contained in a submanifold $\overline{P}_1$ of the maximal codimension $1\leq m_1\leq 2$. Then as $i\rightarrow\infty$, $P_{1,i}$ becomes dense in $\overline{P}_1$, because otherwise, the argument in \cite{mozes_shah_1995} or \cite{Lee-Oh} yields a further infinite subsequence lying in a proper submanifold of $\overline{P}_1$, but it violates the assumption that $\overline{P}_1$ attains the maximal codimension. 
    Similarly, the density ensures that $\overline{P}_1\in \mathscr{P}_{m_1}\subset \mathscr{P}_1\cup\mathscr{P}_2$.

    Furthermore, if the number of candidates in $\{P_i\}$ intersecting $(M\setminus \overline{P}_1)$ is finite, we deduce that $\mathscr{P}$ can be represented by a finite union of submanifolds. Otherwise, we continue to extract an infinite subsequence $\{P_{2,i}\}$ of $\{P_i\}$ meeting $M\setminus\overline{P}_1$, which is contained in a proper submanifold $\overline{P}_2\subset M$ of the maximal codimension $1\leq m_2\leq m_1\leq 2$. Similarly, the maximality makes the limit of $P_{2,i}$ dense in $\overline{P}_2$, and therefore $\overline{P}_2\in\mathscr{P}_{m_2}\subset\mathscr{P}_1\cup\mathscr{P}_2$. 
    
    Finally, since $\mathscr{P}_1\cup\mathscr{P}_2$ is finite, we can only find finitely many closures $\overline{P}_1,\overline{P}_2,\\\cdots,\overline{P}_l$, and the number of elements in $\{P_i\}$ intersecting $M\setminus(\underset{j=1}{\overset{l}{\cup}}\overline{P}_j)$ is finite. Thus, $\mathscr{P}$ is contained in a finite union of proper submanifolds. 
    \end{enumerate}
\end{Ex}

The same method applies to any ambient dimensions, so similarly, we prove by induction that for each $1\leq k\leq n-2$, $\underset{j\leq k}{\cup}\mathscr{P}_j$ is contained in a finite union of proper submanifolds of $M$, this implies Lemma \ref{4.3}.

Now we complete the proof of Lemma \ref{4.2}.
By Lemma \ref{4.3}, all elements in $\mathscr{P}$ lie in a finite union of proper submanifolds of $M$. So there must be a non-empty compact set $K$ in $\Delta\subset\mathbb{H}^n$ away from the fundamental domain of this union, it means that $K$ is disjoint from all $D(\gamma)$ with $\gamma\in\mathscr{L}$, but this contradicts (\ref{star}). Therefore, $\mathscr{L}$ contains an element with a dense $\Gamma$-orbit, and $\mathscr{L}=\mathscr{C}$ follows from the closedness and $\Gamma$-invariance of $\mathscr{L}$.

\end{proof}

Fix a round circle $\gamma\in\mathscr{L}$ that has a dense $\Gamma$-orbit, $\gamma$ can be represented by 
$\underset{i\rightarrow\infty}{\lim}\Lambda(\phi_i\Pi_i\phi_i^{-1})$, where $\phi_i\in F_i(\epsilon_i,R_i)$, as $i\rightarrow \infty$, we have $\epsilon_i\rightarrow 0$ and $R_i\rightarrow\infty$. Let $D_i,\Omega_i$ be the lifts of $S_i,\Sigma_i$ to $B^n$ preserved by $\phi_i\Pi_i\phi_i^{-1}$. We have proved in Section \ref{section2} that after passing to a subsequence, $D_i$ converges to the totally geodesic disc $D=D(\gamma)$. Moreover, 
it follows from (\ref{f}) that \begin{equation*}
    \lim_{i\rightarrow\infty} \int_{\Omega_i\cap B_{R_i}(0)}\big(-(K_{12}+1)+\frac{1}{2}|A|^2+2\pi\sum_jd_{i_j}\delta_{q_{i_j}}\big)dA_h=0.
\end{equation*}
Since $K_{12}\leq -1$, we obtain that \begin{equation}\label{fff}
     \lim_{i\rightarrow\infty} \int_{\Omega_i\cap B_{R_i}(0)}\big(-(K_{12}+1)+\frac{1}{2}|A|^2\big)dA_h=0,
\end{equation}
and \begin{equation*}
    \lim_{i\rightarrow\infty}\int_{\Omega_i\cap B_{R_i}(0)}\sum_jd_{i_j}\delta_{q_{i_j}}dA_h=0.
\end{equation*}
Recall that $Q_i$ represents the set of branch points in $\Omega_i$, the latter equation implies that 
\begin{equation}\label{limitofbranch}
    \#\{Q_i\cap B_{R_i}(0)\}\rightarrow 0\quad\text{as }i\rightarrow\infty.
\end{equation}
So for large enough $i$, $\Omega_i$ has no branch points inside $B_{R_i}(0)$.

\begin{Lemma}\label{4.4}
\emph{There exists a connected component $\Omega^0_i\subset \Omega_i$, so that $\Omega^0_i$ is a disc and it converges smoothly to a totally geodesic hyperbolic disc $\Omega$ with $\partial_\infty\Omega=\gamma$.}
\end{Lemma}  

\begin{proof}
We can explore the convex hulls in the same way as in $\mathsection 3$ of \cite{Calegari-Marques-Neves}, then Proposition 2.5.4 in \cite{Bowditch} and the Morse lemma give rise to a uniform constant $R_0>0$, so that \begin{equation*}
    d_H(C_h(\Lambda(\phi_i \Pi_i\phi_i^{-1}),C_{h_{hyp}}(\Lambda(\phi_i \Pi_i\phi_i^{-1}))\leq R_0,
\end{equation*} where $C_h$ and $C_{h_{hyp}}$ represent the convex hull with respect to metrics $h$ and $h_{hyp}$, respectively. Moreover, \cite{Calegari-Marques-Neves} also proves that
\begin{equation}\label{convexhull}
    \Omega_i\subset C_h(\Lambda(\phi_i \Pi_i\phi_i^{-1}).
\end{equation}
Let $H^r_i$ be the hypersurface in $\mathbb{H}^n$ with the fixed distance $r$ to $D_i$.
By the proof of Lemma \ref{uniqueness}, when $r>\tanh^{-1}\frac{|A|^2_{L^\infty(D_i)}}{2}$, $H^r_i$ is strictly convex and it bounds inside the convex hull of $\Lambda(\phi_i \Pi_i\phi_i^{-1})$, so  \begin{equation*}
    d_H(C_{h_{hyp}}(\Lambda(\phi_i \Pi_i\phi_i^{-1}),D_i)\leq \tanh^{-1}\frac{|A|^2_{L^\infty(D_i)}}{2}.
\end{equation*}  Combining these estimates, we conclude that the Hausdorff distance between $D_i$ and $\Omega_i$ is uniformly bounded.
So there exists $R>0$, for $i>>1$ and generic $r\geq R$, $\Omega_i$ intersects $B_{r}(0)$ by a union of circles. Then we can slightly perturb $R_i$ so that $\Omega_i\cap B_{R_i}(0)$ is a union of circles.

Let $\Omega_i^0$ be a component of $\Omega_i\cap B_{R_i}(0)$ that intersects $B_{R}(0)$, by (\ref{limitofbranch}), for sufficiently large $i$, $\Omega_i^0$ is free of branch points, so it is embedded in $B^n$.
We claim that it is a disc. Otherwise, if $\Omega_i^0$ were an annulus, then we could find a larger ball $B_{R_i'}(0)$ with some $R_i'>R_i$ whose boundary met tangentially with $\Omega_i^0$ at some point. However, the convexity of $\partial B_{R_i'}(0)$ and the minimality of $\Omega_i^0$ contradict the maximum principle. Therefore, $\Omega_i^0$ is a disc provided that $i$ is large enough. 
Furthermore, The small total curvature estimates based on (\ref{fff}) imply that \begin{equation}
\underset{i\rightarrow\infty}{\lim\sup}\{|K_{12}(x)+1|+\frac{1}{2}|A(x)|^2:x\in\Omega_i^0\}=0.
\end{equation}
From the standard compactness theorem for minimal surfaces with a uniform bound on second fundamental form, after passing to a subsequence, $\Omega_i^0$ converges smoothly to a minimal disc $\Omega$ in $(B^n,h)$. Moreover, $\Omega$ is totally geodesic and it has sectional curvature equal to $-1$. 

It remains to show that $\partial_\infty\Omega=\gamma$. 
Take a sequence $x_i\in \Omega_i^0$ that converges to $x\in \Omega$, and take $y_i\in \Lambda(\phi_i\Pi_i\phi_i^{-1})$. Let $\alpha_i$ be the geodesic arc in $(B^n,h)$ connecting $x_i$ to $y_i$, and let $\beta_i$ be the geodesic arc in $\Omega_i$ connecting $x_i$ to $y_i$. 
Due to (\ref{convexhull}) and Proposition 2.5.4 in \cite{Bowditch}, we can find a uniform number $r>0$, such that $\beta_i$ is contained in the $r$-neighborhood of $\alpha_i$. 
Additionally, since $\Omega$ is totally geodesic,  both $\alpha_i$ and $\beta_i$ converge to the same geodesic arc in $\Omega$ that connects $x$ to some $y\in\partial_\infty \Omega$. Then $y_i$ converges to $y$ on $S^{n-1}_\infty$. As a consequence, $\partial_\infty \Omega\subset \gamma$, since $\partial_\infty\Omega$ is a circle, it coincides with $\gamma$. 

\end{proof}

As defined in $\mathsection 5$ of \cite{Calegari-Marques-Neves}, $S_1^D(M)$ and $S_1^\Omega(M)$ denote the projections of the circle bundles of $D$ and $\Omega$ to the unit tangent bundles of $M$ with resprect to to $h_{hyp}$ and $h$, respectively. 
Since $S_1^D(M)$ is dense in the unit tangent bundle $S_1M(h_{hyp})$, then via the homeomorphism from $S_1M(h_{hyp})$ to $S_1M(h)$ that maps geodesics to geodesics \cite{Gromov2000}, we obtain that $S_1^\Omega(M)$ is dense in $S_1M(h)$. Thus for any $(x,v)\in S_1M(h)$, there is a sequence $\{\psi_i(\Omega)\}_i$, $\psi_i\in\Gamma$, converging to a totally geodesic hyperbolic disc $\Omega_{(x,v)}$ in $(B^n,h)$, whose projection in $M$ contains a geodesic passing through $x$ with direction $v$. According to the ergodicity of the 2-frame flows on the negatively curved manifolds of arbitrary odd dimensions ($\mathsection 4$ in \cite{Brin-Gromov}), the set of totally geodesic hyperbolic discs is dense in $Gr_2(M)$.

Therefore, if the equality (\ref{area}) holds, then $(M,h)$ is hyperbolic, and it is isometric to $(M,h_{hyp})$ due to Mostow rigidity theorem.

\subsection{Proof of rigidity in (b)}
First of all, if the metric $h$ has sectional curvature less than or equal to $-1$, then $\Pi\in S(M,\floor*{L},\epsilon)$ implies that $\text{area}_h(\Pi)\leq 4\pi(L-1)$ because of the generalized Gauss equation (\ref{generalizedgauss}), thus, $E(h)\geq 2= E(h_{hyp})$.

Next, suppose $E(h)=2$. Assume that there exists $\eta>0$, such that for all $L>0$ and all increasing sequence $\{k_i\}\subset\mathbb{N}$, the condition $\Pi\in G(M,\floor*{(1+\eta)L},\frac{1}{k_i})$ must produce that $\text{area}_h(\Pi)\leq 4\pi(L-1)$.
As a result,\begin{equation*}
    E(h)\geq\underset{L\rightarrow\infty}{\lim\inf}\frac{\ln{\#G(M,\floor*{(1+\eta)L},\frac{1}{k_i})}}{L\ln{L}}\geq 2(1+\eta),
\end{equation*}
which violates the assumption. Therefore, there exists an increasing sequence $\{k_i\}\subset\mathbb{N}$, a sequence of integers $\{g_i\}$ and $\Pi_i\in G(M,g_i,\frac{1}{k_i})$, so that \begin{equation*}
    \text{area}_h(\Pi_i)\leq 4\pi((1-\frac{1}{i})g_i-1).
\end{equation*}
Let $\Sigma_i$ and $S_i$ be the minimal surfaces that minimizes area in the homotopy class $\Pi_i$ with respect to metrics $h$ and $h_{hyp}$, respectively. Then from the inequality above and Theorem \ref{5.1},\begin{equation*}
     1\geq\underset{i\rightarrow\infty}{\lim\sup}\dfrac{\text{area}_h(\Sigma_i)}{\text{area}(S_i)}\geq \underset{i\rightarrow\infty}{\lim\inf}\dfrac{\text{area}_h(\Sigma_i)}{\text{area}(S_i)}\geq \underset{i\rightarrow\infty}{\lim\inf}\frac{4\pi((1-\frac{1}{i})g_i-1)}{4\pi(g_i-1)}=1.
\end{equation*}
The equality holds if and only if $h$ is hyperbolic, and thus it is isometric to $h_{hyp}$.

\section{Extension to locally symmetric spaces of rank one}\label{section_symmetricspace}
In this section, we extend Theorem \ref{main} to locally symmetric spaces of rank one apart from real hyperbolic manifolds.

\subsection{Definition of entropy}
Let $M$ be a closed $2n$ ($4n, 16$)-dimensional complex hyperbolic manifold (quaternionic hyperbolic manifold, Cayley plane, respectively), where $n\geq 2$. Then its sectional curvature is between $-4$ and $-1$. In the Siegel domain model of $\mathbb{H}_{\mathbb{C}}^n$ ($\mathbb{H}_{\mathbb{H}}^n, \mathbb{H}_{Ca}^2$), its boundary is identified with the one point compactification of the Heisenberg group. A totally geodesic disc in $\mathbb{H}_{\mathbb{C}}^n$ ($\mathbb{H}_{\mathbb{H}}^n, \mathbb{H}_{Ca}^2$) with constant sectional curvature $-1$ is called \emph{totally real} and is isometric to $\mathbb{H}_{\mathbb{R}}^2$, whose boundary is a \emph{real circle}. 

A $K$-quasiconformal map on $\partial\mathbb{H}_{\mathbb{C}}^n=S_\infty^{2n-1}$ ($\partial\mathbb{H}_{\mathbb{H}}^n=S_\infty^{4n-1}, \partial\mathbb{H}_{Ca}^2=S_\infty^{15}$) is defined as in Section \ref{1.1} with respect to Carnot-Carath{\'e}odory metric (see \cite{Koranyi-Reimann}).
In particular, for quaternionic hyperbolic spaces and Cayley plane, \cite{Pansu} points out that any quasiconformal maps are actually conformal. Then the quasi real circles are real circles, each of them determines a unique totally geodesic totally real disc. Therefore, if the manifold $M$ admits quaternionic hyperbolic or Cayley metric, then $S(M,g)=S(M,g,0)$, and let $S(M)=\underset{g\geq 2}{\cup}S(M,g)$. For any metric $h$ on $M$, the corresponding entropy is redefined as \begin{equation*}
     E(h)=\underset{L\rightarrow\infty}{\lim\inf}\,\dfrac{\ln\#\{\text{area}_h(\Pi)\leq 4\pi(L-1):\Pi\in S(M)\}}{L\ln L}.
\end{equation*}
Besides, if $M$ admits a complex hyperbolic metric, we still adopt the definition in (\ref{eh}). Then the main theorem related to the locally symmetric spaces is stated as follows.
\begin{Th}\label{main_sym}
Let $(M,h_{sym})$ be a closed locally symmetric space of rank one. And let $h$ be another metric on $M$. If the sectional curvature of $h$ is pointwisely less than or equal to that of the locally symmetric metric, then \begin{equation*}
        E(h)\geq E(h_{sym})=2,
    \end{equation*}
    If the equality holds, then $h$ is isometric to $h_{sym}$.
\end{Th}

\subsection{Entropy of locally symmetric metric}
Hamenst{\"a}dt \cite{Hamenstadt} proved the existence of the surface subgroup of cocompact lattice in any simple rank one Lie group of noncompact type distinct from $SO(2m,1)$. From the prospective of geometry, let $M$ be any closed locally symmetric space except an even-dimensional real hyperbolic manifold,
then for sufficiently small $\epsilon$, there exists an essential surface $\Sigma_\epsilon\subset M$, which is $(1+o_\epsilon(1))$-quasigeodesic. As argued in Section \ref{section_lower}, it is associated with a surface subgroup in $S(M,g,\epsilon)$ for the complex hyperbolic case, or $S(M,g)$ for the quaternionic hyperbolic and Cayley case.
Moreover, $s(M,g,\epsilon)$ (or $s(M,g)$) is also bounded below by $(c_1g)^{2g}$. On the other hand, since the power of the upper bound of $s(M,g)$ in Section \ref{section_upper} only depends on the topology of closed surfaces, the upper 
bound $(c_2g)^{2g}$ also holds after modifying the coefficient $c_2$.

Firstly, if $(M,h_{sym})$ is quaternionic hyperbolic or Cayley hyperbolic, it's not hard to show $E(h_{sym})=2$ based on the above estimates.

If $(M,h_{sym})$ is complex hyperbolic, however, it requires more discussion on the second fundamental form as in Section \ref{section2}. Lemma \ref{lem_barrier} is still true because of the following fact. 
Let $\gamma$ be the image of a real circle by an $(1+\epsilon)$-quasiconformal map on $S^{2n-1}_\infty$. Then there exist $\epsilon'=o_\epsilon(1)$ and a real circle $c\subset S^{2n-1}_\infty$, such that the $\epsilon'$-neighborhood of $c$ in $S^{2n-1}_\infty$, denoted by $N$, contains $\gamma$. Any real circle in $\partial N$ bounds a totally geodesic totally real disc in $\mathbb{H}_{\mathbb{C}}^n$, then there is a hypersurface $T\subset \mathbb{H}_{\mathbb{C}}^n$ homeomorphic to $D^2\times S^{2n-3}$ asymptotic to $\partial N$, the diameter of $T$ converges to zero on compact sets as $\epsilon$ goes to zero.
Moreover, two of the principal curvatures of $T$ are zero, the others are at least $\dfrac{1}{\tanh 2r(x)}>0$, where $r(x)$ is the Euclidean radius of $T$ centered at $x$, therefore the convex hull of $\gamma$ lies in $T$.
Regarding the proof of Lemma \ref{uniqueness}, since all principal curvatures of the equidistant hypersurfaces satisfy the Riccati equations, 
the comparison theorem associated with Riccati equations (see \cite{Eschenburg}) ensures the two-convexity of each hypersurface. Moreover, the compactness theorem of the quasiconformal maps on the Heisenberg group can be found in \cite{Koranyi-Reimann}. So Lemma \ref{uniqueness} extends easily to the complex hyperbolic case. 
Likewise, since the manifold has pinched sectional curvature between $-4$ and $-1$, there are analogues of the inequalities (\ref{area_bound}) and (\ref{lower_bound}) (see \cite{Bangert-Lang}). Let $S\subset M$ be a surface that minimizes the area in its homotopy class contained in $S_\epsilon(M)$.
Then the convergence of area-minimizing surfaces mod 2 and the uniqueness indicate the absence of branch points on $S$, as well as the property that $|A|^2_{L^\infty(S)}\rightarrow 0$ as $\epsilon\rightarrow 0$. Following the computation in Section \ref{subsection_b}, we deduce that $E(h_{sym})=2$ for the complex hyperbolic case.

\subsection{Proof of rigidity}
Let $h$ be a metric on $M$ with sectional curvature less than or equal to that of the symmetric metric, then it follows from the generalized Gauss equation that $E(h)\geq 2$. 

To deduce the rigidity of Theorem \ref{main_sym}, the key idea is to apply the hyperbolic rank rigidity theorem established by Hamenst{\"a}dt in \cite{hamenstadt91}. Suppose that $N$ is a closed manifold whose sectional curvature is less than or equal to $-1$. The \emph{hyperbolic rank at $v\in T^1N$} is the dimension of the space generated by all parallel transports $J$ along geodesic $\gamma_v=\exp(tv)$ such that $J(t)\perp\gamma_v'(t)$ and $J(t),\gamma_v'(t)$ span a plane of sectional curvature $-1$. The \emph{hyperbolic rank of $N$} is the minimum of hyperbolic rank at all $v\in T^1N$. Hamenst{\"a}dt's theorem states that if such manifold $N$ has hyperbolic rank at least $1$, then it must be locally symmetric, namely a compact quotient of $\mathbb{H}_{\mathbb{R}}^n, \mathbb{H}_{\mathbb{C}}^n, \mathbb{H}_{\mathbb{H}}^n$ or $\mathbb{H}_{Ca}^2$.

Therefore, It suffices to check that the hyperbolic rank of $M$ is positive. 
Since the technic in \cite{Calegari-Marques-Neves} proving (\ref{star}), Shah's density lemma in \cite{Shah2}, as well as the equidistribution theorem by Mozes and Shah \cite{mozes_shah_1995} all apply to locally symmetric spaces, an analogue of Lemma \ref{4.2} can be deduced in the same way. Furthermore, repeating the same proof of Lemma \ref{4.4}, we obtain a totally geodesic disc $\Omega$ in $(B^m,h)$ of section curvature $-1$, whose limit set has a dense $\Gamma$-orbit on the set of all real circles in $S^{m-1}_\infty$, where $m$ is the dimension of $M$. This result in tandem with Gromov's geodesic rigidity theorem \cite{Gromov2000} implies that every geodesic along $v\in T^1M$ is contained in a closed totally geodesic submanifold of dimension $2\leq k\leq n$ and with sectional curvature $-1$. Therefore $(M,h)$ has positive hyperbolic rank, and finally the rigidity result of Theorem \ref{main_sym} follows from Hamenst{\"a}dt's hyperbolic rank rigidity theorem mentioned above.

\section{Extension to hyperbolic 3-manifolds of finite volume}\label{section_cusp}
In this section, we compute the entropy corresponding to the noncompact hyperbolic 3-manifolds of finite volume, which have finitely many cusps.
Suppose $(M,h_{hyp})$ is a hyperbolic 3-manifold with $k$ cusps, then $M$ can be realized by the interior of a compact hyperbolic manifold whose boundary consists of $k$ flat tori, we also denote the compact manifold with boundary by $M$.
As before, a closed surface immersed in $M$ is essential if the immersion is $\pi_1$-injective. Additionally, a noncompact surface (or a compact surface with boundary) is said to be \emph{essential} if the immersion is $\pi_1$-injective and $\pi_1$-injective relative to the boundary.
By Lemma 2.1 in \cite{Hass-Rubinstein-Wang}, any $\pi_1$ injective noncomapct surface with genus at least $2$ is also essential.
When $S\subset M$ is an essential surface, the image of $\pi_1(S)$ in $\pi_1(M)$ is called a \emph{surface subgroup}. 
Let $S_j(M,g)$ denote the set of surface subgroups up to conjugacy so that the corresponding surfaces have genus at most $g$ and at most $j$ simple closed curves on the boundary counting multiplicity, and let $S_j(M,g,\epsilon)\subset S_j(M,g)$ consist of the conjugacy classes whose limit sets are $(1+\epsilon)$-quasicircles. Moreover, \begin{equation*}
    S_j(M,\epsilon)=\underset{g\geq 2}{\cup}S_j(M,g,\epsilon).
\end{equation*} 
Given an arbitrary Riemannian metric $h$ on $M$, we define the entropy of $h$ on $M$. 
\begin{equation}\label{cusp_entropy}
    E_j(h)=\underset{\epsilon\rightarrow 0}{\lim}\,\underset{L\rightarrow\infty}{\lim\inf}\,\dfrac{\ln\#\{\text{area}_h(\Pi)\leq 4\pi(L-1):\Pi\in S_j(M,\epsilon)\}}{L\ln L}.
\end{equation}

The statement of the theroem is the following.
\begin{Th}\label{main_cusp}
Let $M$ be a hyperbolic 3-manifolds with $k$ cusps, then for any $j\in\mathbb{N}$, we have \begin{equation*}
        E_j(h_{hyp})=2.
    \end{equation*}
\end{Th}

\begin{Rem}
To evaluate the entropy of $M$ with respect to other metrics, we need additional existence arguments of least area surfaces, which is still an open problem. 
\end{Rem}

\subsection{Existence of minimal surfaces}
We've seen that for closed hyperbolic manifolds,  Schoen-Yau \cite{Schoen-Yau} and Sacks-Uhlenbeck \cite{Sacks-Uhlenbeck} proved that every surface subgroup produces a least-area surface in the homotopy class. However, the argument fails to hold for some noncompact ambient 3-manifolds (see Example 6.1 in \cite{Hass-Scott}).
In this section, we only list the existence results for hyperbolic 3-manifolds with finitely many cusps. 
First, it was stated by Hass-Rubinstein-Wang \cite{Hass-Rubinstein-Wang} and Ruberman \cite{Ruberman} that in a cusped hyperbolic 3-manifold, any noncompact  essential surface with genus at least $2$ can be homotoped to a least-area surface. 
Then in \cite{Collin-Hauswirth-Mazet-Rosenberg} and \cite{Collin2019CorrigendumT}, Collin-Hauswirth-Mazet-Rosenberg proved the existence of closed essential minimal surfaces embbedded in such manifolds. 
Later on, Huang-Wang addressed the question for immersed essential surfaces in \cite{Huang-Wang}, they showed that any immersed essential surface in a cusped hyperbolic 3-manifold with genus at least $2$ can be homotoped into an area-minimizer. 
Therefore, the entropy (\ref{cusp_entropy}) of cusped hyperbolic 3-manifolds $(M,h_{hyp})$ can be approximated by counting the least-area surfaces up to homotopy, and we'll estimate the upper bound and lower bound of $\#S_j(M,g,\epsilon)$ associated with $h_{hyp}$ and prove the theorem.

\subsection{The upper bound}
\subsubsection{Counting closed minimal surfaces}
First of all, notice that $S_0(M,g,\epsilon)$ consists of all the surface subgroups corresponding to closed immersed least-area surfaces, we estimate the cardinality of  $S_0(M,g,\epsilon)$ first.
By \cite{Rubinstern}, any closed immersed surface $\Sigma\subset M$ with $|A|^2_{L^\infty(\Sigma)}<2$ cannot have \emph{accidental parabolics}, meaning that there's no essential non-peripheral closed curve in $\Sigma$ that can be homotoped into the cusps in $M$, here a curve is said to be \emph{non-peripheral} if it does not bound an annulus. For sufficiently small $\epsilon$, 
the argument in Section \ref{section2} indicates that the minimal surface $S$ corresponding to an arbitrary element in $S(M,g,\epsilon)$ admits a small upper bound on the norm squared of the second fundamental form of $S$, thus it preserves parabolicity.

Then we claim that such a closed immersed least-area surface $S\subset M$ with genus $g\geq 2$ cannot go deeply into the cusps $T_1,\cdots,T_k$. To see this, note that each cusp of $M$ is homeomorphic to $T^2\times [0,\infty)$, where $T^2$ is a torus, then we choose the maximal cusped region of $M$, such that each cusp region can be represented by $T_i\times [r_0,\infty)$, and they are disjoint from each other.
Suppose $S$ intersects some $T_i\times [r_0,\infty)$, and we pick an arbitrary closed curve $\alpha$ lying in $S\cap (T_i^2\times \{r_0\})$. Firstly, suppose that $\alpha$ is essential. It must be non-peripheral, because otherwise, $S$ cannot be lifted to a closed surface embedded in a finite cover of $M$, which contradicts the fact that cusped hyperbolic 3-manifolds are locally extended residually finite (LERF, see Theorem 2.2 of \cite{Huang-Wang}).
Moreover, the lack of accidental parabolics forces the homotopy class of $\alpha$ mapping to the trivial element in $\mathbb{Z}\oplus\mathbb{Z}$. But this is impossible since $S$ is an essential surface in $M$. So we prove that $\alpha$ must be non-essential, namely, it bounds an embedded disc $D\subset S$. If $D$ goes deeply into the cusp, assume that $R_S$ is the maximum number of $r$ such that $D$ intersects $T_i^2\times \{r\}$, then according to the coarea formula (\cite{Wang09}, Lemma 4.1),
\begin{equation*}
    \text{area}(D)=\int_{r_0}^\infty \int_{D\cap (T_i\times\{r\})}\frac{1}{\cos\theta}dldr\geq \int_{r_0}^{R_S} l(r)dr,
\end{equation*}
where for any point $x\in D$, $\theta(x)$ is the angle between the tangent plane to $D$ at $x$, and the radial geodesic emanating from $x$ and perpendicular to $T_i\times\{r_0\}$, moreover, $l(r)$ represents the length of $D\cap (T_i\times\{r\})$.
However, we can choose a number $R_0<R_S$ that is independent of $S$, and replace $D$ by a disc $\bar{D}$ in $T_i^2\times [r_0,R_0]$ with $\partial \bar{D}=\alpha$, which has smaller area. To find $\bar{D}$, we only need to shrink $D\cap (T_i\times\{r\})$ to a shorter closed curve with length $\bar{l}(r)< \frac{1}{2}l(\frac{R_S}{R_0}r)$ for any $r\in[r_0,R_0]$, and make sure the corresponding angle $\bar{\theta}$ at each point is at most $\frac{\pi}{3}$. Thus, $\bar{D}$ is contained in $T_i\times [r_0,R_0]$, and \begin{align*}
    \text{area}(\bar{D})&=\int_{r_0}^{R_0}\int_{\bar{D}\cap (T_i\times\{r\})}\frac{1}{\cos\bar{\theta}}dldr\leq 2\int_{r_0}^{R_0}\bar{l}(r)dr< \int_{r_0}^{R_0} l(\frac{R_S}{R_0}r)dr\\
    &=\frac{R_0}{R_S}\int_{\frac{R_S}{R_0}r_0}^{R_S}l(r)dr<\int_{r_0}^{R_S}l(r)dr.
\end{align*}
So $\text{area}(\bar{D})<\text{area}(D)$, then let $\bar{S}=(S\setminus D)\cup \bar{D}$, we get $\text{area}(\bar{S})<\text{area}(S)$, but this cannot happen because $S$ has the least area. In consequence, all closed surfaces immersed in $M$ that minimize area in the homotopy classes are disjoint from $T_i\times (R_0,\infty)$.
In other words, there exists some $s>0$ depending only on $M$, such that all closed immersed minimal surfaces associated with elements in $S_0(M,g,\epsilon)$ are contained in the thick part $M_{[2s,\infty)}$, whose injectivity radius is at least $s$.
Therefore, it follows from Section \ref{section_upper} that \begin{equation}\label{s_0}
    \#S_0(M,g,\epsilon)=s(M_{[2s,\infty)},g,\epsilon)\leq (c_2g)^{2g},
\end{equation}
where $c_2>0$ depends only on $M$ and $\epsilon$.

\subsubsection{Counting noncompact minimal surfaces}
Next, we estimate $\#S_j(M,g,\epsilon)$ for $j\geq 1$. For sufficiently small $\epsilon$, let $S$ be a minimal surface corresponding to an element in $S_j(M,g,\epsilon)\setminus S_0(M,g,\epsilon)$, where $g\geq 2$ and $j\geq 1$, then $S$ can be seen as a compact surface with $j$ boundary curves counting multiplicity, each of which is a simple closed curve $C$ in one of the tori $T_i$, the homotopy class of $C$ can be identified with a slope in $\mathbb{Q}\cup\{\infty\}$.
Due to \cite{Hass-Rubinstein-Wang}, for each $g$, there is a finite number $N(g)>0$, such that for each boundary torus, the number of slopes that can be realized by boundary curves of immersed essential surfaces in $M$ with genus at most $g$ is bounded from above by $N(g)$. Moreover, $N(g)$ grows at most quadratically in $g$. Cutting off each boundary curve of $S$ and filling it with a disc, we obtain a closed surface in $M$ associated with an element in $S_0(M,g,\epsilon)$ up to homotopy. 
As a result, $\#S_j(M,g,\epsilon)$ is bounded by a constant depending on $k, j$ and $\epsilon$. More precisely, \begin{equation*}
    \#S_j(M,g,\epsilon)\leq \underset{i=1}{\overset{j}{\sum}}(kN(g))^i\#S_0(M,g,\epsilon)\leq (c_2'g)^{2g},
\end{equation*}
where $c_2'>0$ depends only on $M$ and $\epsilon$, and the last inequality follows from (\ref{s_0}).

\subsection{The lower bound}
Recently, Kahn-Wright \cite{kahn-wright2020} proved that when $M=\mathbb{H}^3/\Gamma$ is noncompact with finite volume, for any sufficiently small $\epsilon>0$, there exists $(1+\epsilon)$-quasi-Fuchsian surface subgroups of $\Gamma$. And the construction gives rise to an essential surface $\Sigma_\epsilon$, which is also sufficiently well-distributed and $(1+\epsilon)$-quasigeodesic. 
So as defined in the compact case, we let $G(M,g,\epsilon)$ be the subset of $S_0(M,g,\epsilon)$ consisting of homotopy classes of finite covers of $\Sigma_\epsilon$ with genus at most $g$, then we also have \begin{equation*}
    \#S_j(M,g,\epsilon) \geq \#S_0(M,g,\epsilon)\geq \#G(M,g,\epsilon)\geq (c_1g)^{2g},\quad \forall j\in\mathbb{N}
\end{equation*}
where $c_1>0$ depends only on $M$ and $\epsilon$.

\subsection{Proof of theorem}
If a surface subgroup $\Pi<\Gamma$ has $i$ cusps, then from \cite{Fang96}, \begin{equation*}
    \text{area}(\Pi)=4\pi(g-1)+2\pi i-\frac{1}{2}\int |A|^2dA.
\end{equation*}
Since it is proved in Section \ref{section2} that $|A|^2=o_\epsilon(1)$, and $i\leq j$ is uniformly bounded, then for any $\eta>0$, the following conclusions still hold for sufficiently large $L$ and sufficiently small $\epsilon$. For $\Pi\in S_j(M,\epsilon)$, if it satisfies $\text{area}(\Pi)\leq 4\pi(L-1)$, then we have $\Pi\in S_j(M,\floor*{(1+\eta)L},\epsilon)$. On the other hand, if $\Pi\in S_j(M,\floor*{(1-\eta)L},\epsilon)$,
then $\text{area}(\Pi)\leq 4\pi(L-1)$.
Therefore the previous argument shows that $E_j(h_{hyp})=2$.

\bibliographystyle{plain} 
\bibliography{counting}   
\end{document}